\documentclass[11pt]{amsart}
\usepackage{amssymb, pstricks,pst-plot}
\input xy
\xyoption{all}

\definecolor{OrangeRed}{cmyk}{0,0.6,1,0}
\definecolor{DarkBlue}{cmyk}{1,1,0,0.20}
\definecolor{Black}{cmyk}{0,0,0,1}
\definecolor{Violet}{cmyk}{0.79,0.88,0,0}
\definecolor{myblue}{rgb}{0.66,0.78,1.00}

\parskip=\smallskipamount

\newtheorem{theorem}{Theorem}[section]

\newtheorem{corollary}[theorem]{Corollary}
\newtheorem{proposition}[theorem]{Proposition}

\theoremstyle{definition}
\newtheorem{definition}[theorem]{Definition}

\newtheorem{remark}[theorem]{Remark}

\newcommand{\C}{\mathbb{C}}

\newcommand{\N}{\mathbb{N}}

\renewcommand{\P}{\mathbb{P}}
\newcommand{\R}{\mathbb{R}}
\newcommand{\T}{\mathrm{T}}

\newcommand{\I}{\mathrm{i}}

\newcommand{\cA}{\mathcal{A}}

\newcommand{\cC}{\mathcal{C}}

\newcommand{\cK}{\mathcal{K}}

\newcommand{\cO}{\mathcal{O}}

\newcommand{\cU}{\mathcal{U}}

\renewcommand{\d}{\mathrm{d}}

\newcommand\VT{\mathrm{VT}}

\newcommand\holo{holomorphic} 

\newcommand\nbd{neighborhood}
\newcommand\wt{\widetilde}

\def\dibar{\overline\partial}
\def\bs{\backslash}

\newcommand\CAP{{\rm CAP}}
\newcommand\PCAP{{\rm PCAP}}
\newcommand\BOP{{\rm BOP}}
\newcommand\POP{{\rm POP}}

\numberwithin{equation}{section}

\begin{document}
\title[Invariance of the parametric Oka property]
{Invariance of the parametric Oka property}
\author{Franc Forstneri\v c}
\address{Faculty of Mathematics and Physics, University of Ljubljana, 
and Institute of Mathematics, Physics and Mechanics, Jadranska 19, 
1000 Ljubljana, Slovenia}
\email{franc.forstneric@fmf.uni-lj.si}
\thanks{Supported by the research program P1-0291, Republic of Slovenia.}

%
%
%
%
\date{\today}

%
%
%
%
\begin{abstract}
Assume that $E$ and $B$ are complex manifolds and that 
$\pi\colon E\to B$ is a holomorphic Serre fibration
such that $E$ admits a finite dominating
family of holomorphic fiber-sprays over a small neighborhood of any 
point in $B$. We show that the parametric Oka property (POP) 
of $B$ implies POP of $E$; conversely, POP of $E$ implies POP of $B$
for contractible parameter spaces. This follows from 
a parametric Oka principle for holomorphic liftings 
which we establish in the paper.
\end{abstract}
\maketitle

\centerline{\em Dedicated to Linda P.\ Rothschild}
\smallskip

\section{The Oka properties}
\label{Sec:Oka}
The main result of this  paper is that a subelliptic holomorphic
submersion $\pi\colon E\to B$ between (reduced, paracompact) 
complex spaces satisfies the {\em parametric Oka property}.
{\em Subellipticity} means that $E$ admits a finite dominating
family of holomorphic fiber-sprays over a neighborhood of any 
point in $B$ (Def.\ \ref{def:SSF}). 
The conclusion means that for any Stein source space $X$,
any compact Hausdorff space $P$ (the parameter space), 
and any continuous map $f\colon X\times P\to B$ which is $X$-holomorphic 
(i.e., such that $f_p=f(\cdotp,p)\colon X\to B$ is holomorphic 
for every $p\in P$), a continuous lifting $F\colon X\times P\to E$ of $f$ 
(satisfying $\pi\circ F=f$) can be homotopically deformed 
through liftings of $f$ to an $X$-holomorphic lifting.  
(See Theorem  \ref{SES:lifting-maps} for a precise statement.)

\vskip -5mm
\[
	\xymatrix{ & E \ar[d]^{\pi} \\ 
	           X\times P \ar[r]^{f} \ar[ur]^{F} & B }
\]
\vskip 2mm

The following result is an easy consequence. 
Suppose that $E$ and $B$ are complex manifolds and
that $\pi\colon E\to B$ is a subelliptic submersion which 
is also a Serre fibration (such map is called 
a {\em subelliptic Serre fibration}),  or is a 
holomorphic fiber bundle whose fiber satisfies the parametric Oka property.
Then the parametric Oka property passes up from the base 
$B$ to the total space $E$; it also passes down from 
$E$ to $B$ if the parameter space $P$ is contractible, or if  $\pi$ is a 
weak homotopy equivalence (Theorem \ref{SES:ascend-descend}). 

We begin by recalling the relevant notions. 
Among the most interesting phenomena in complex geometry are,
on the one hand, {\em holomorphic rigidity}, commonly expressed
by Kobayashi-Eisenman hyperbolicity; and, on the
other hand, {\em holomorphic flexibility}, a term introduced
in \cite{FF:flexibility}.
While Kobayashi hyperbolicity of a complex manifold $Y$ implies 
in particular that there exist no nonconstant holomorphic maps $\C\to Y$,
flexibility of $Y$ means that it admits many nontrivial holomorphic
maps $X\to Y$ from any Stein manifold $X$; in particular, from 
any Euclidean space $\C^n$. 

The most natural flexibility properties are 
the {\em Oka properties} which originate in the seminal works 
of Oka \cite{Oka1} and Grauert \cite{Grauert2,Grauert3}.
The essence of the classical {\em Oka-Grauert principle} is 
that a complex Lie group, or a complex homogeneous manifold,
$Y$, enjoys the following:

\smallskip
\noindent \textbf{Basic Oka Property (BOP) of $Y$}: 
{\em Every continuous map $f \colon X\to Y$ from a 
Stein space  $X$ is homotopic to a holomorphic map. 
If in addition $f$ is holomorphic on (a neighborhood of) 
a compact $\cO(X)$-convex subset $K$ of $X$, and
if $f|_{X'}$ is holomorphic on a closed complex subvariety $X'$
of $X$, then there is a homotopy $f^t\colon X\to Y$ $(t\in [0,1])$ 
from $f^0=f$ to a holomorphic map $f^1$ 
such that for every $t\in [0,1]$, $f^t$ is holomorphic and 
uniformly close to $f^0$ on $K$, and $f^t|_{X'}=f|_{X'}$.} 
\smallskip

All complex spaces in this paper are assumed to be reduced and paracompact.
A map is said to be holomorphic on a compact subset $K$
of a complex space $X$ if it is holomorphic in an open neighborhood
of $K$ in $X$; two such maps are identified if they
agree in a (smaller) neighborhood of $K$; for a family of maps,
the neighborhood should be independent of the parameter.

When $Y=\C$, BOP combines the Oka-Weil approximation theorem 
and the Cartan extension theorem.
BOP of $Y$ means that, up to a homotopy obstruction, the same
approximation-extension result holds for holomorphic maps $X\to Y$
from any Stein space $X$ to $Y$. 

Denote by $\cC(X,Y)$  (resp.\ by $\cO(X,Y)$) the space of all 
continuous (resp.\ holomorphic) maps $X\to Y$, endowed with the topology
of uniform convergence on compacts. We have a natural inclusion
\begin{equation}
\label{incl}
		\cO(X,Y) \hookrightarrow \cC(X,Y).
\end{equation}
BOP of $Y$ implies that every connected component of $\cC(X,Y)$ 
contains a component of $\cO(X,Y)$. By \cite[Theorem 5.3]{FF:CAP}, 
BOP also implies the following

\smallskip
\noindent 
{\bf One-parametric Oka Property:} 
{\em A path $f\colon [0,1] \to \cC(X,Y)$ such that
$f(0)$ and $f(1)$ belong to $\cO(X,Y)$ can be deformed, with fixed ends at $t=0,1$,
to a path in $\cO(X,Y)$. Hence (\ref{incl}) induces a bijection 
of the path connected components of the two spaces.}
\smallskip

$Y$ enjoys the {\bf Weak Parametric Oka Property} if 
for each finite polyhedron $P$ and subpolyhedron $P_0\subset P$,
a map $f\colon P\to \cC(X,Y)$ such that $f(P_0)\subset \cO(X,Y)$ can be deformed
to a map $\wt f\colon P\to \cO(X,Y)$ by a homotopy that is fixed on $P_0$:
\vskip -3mm
\[
	\xymatrix{ 
	  P_0  \ar[d]_{incl}  \ar[r] & \cO(X,Y) \ar[d]^{incl}  \\
	  P    \ar[r]^{\!\!\!\!\!\! f}  \ar[ur]^{\wt  f}	 & \cC(X,Y) } 
\]

\smallskip
\noindent
This implies that (\ref{incl}) is a weak homotopy equivalence 
\cite[Corollary 1.5]{FPrezelj:OP1}.

\begin{definition}
\label{POP}
\textbf{(Parametric Oka Property (POP))}
Assume that $P$ is a compact Hausdorff space and that $P_0$ is a closed subset of $P$.
A complex manifold $Y$ enjoys {\rm POP} for the pair $(P,P_0)$ if the following holds.
Assume that $X$ is a Stein space, $K$ is a compact $\cO(X)$-convex 
subset of $X$, $X'$ is a closed complex subvariety of $X$, and 
$f\colon X\times P\to Y$ is a continuous map such that 
\begin{itemize} 
\item[(a)]   the map $f_p=f(\cdotp,p)\colon X\to Y$ is holomorphic 
for every $p\in P_0$, and
\item[(b)]  $f_p$ is holomorphic on $K\cup X'$ for every $p\in P$. 
\end{itemize}
Then there is a homotopy $f^t\colon X\times P\to Y$ $(t\in [0,1])$ 
such that $f^t$ satisfies properties (a) and (b) above for all $t\in[0,1]$, and also
\begin{itemize} 
\item[(i)]    $f^1_p$ is holomorphic on $X$ for all $p\in P$, 
\item[(ii)]  $f^t$ is uniformly close to $f$ on $K\times P$ for all $t\in [0,1]$, and 
\item[(iii)]  $f^t=f$ on $(X\times P_0)\cup (X'\times P)$ for all $t\in [0,1]$.
\end{itemize}
The manifold $Y$ satisfies \POP\ if the above holds for each pair 
$P_0\subset P$ of compact Hausdorff spaces.
Analogously we define \POP\ for sections of a holomorphic map $Z\to X$. 
\qed \end{definition}

Restricting POP to pairs $P_0\subset P$ consisting of finite polyhedra
we get Gromov's Ell$_\infty$ property \cite[Def.\ 3.1.A.]{Gromov:OP}.
By Grauert, all complex homogeneous manifolds enjoy 
POP for finite polyhedral inclusions $P_0\subset P$ \cite{Grauert2,Grauert3}.
A weaker sufficient condition, called {\em ellipticity} 
(the existence of a dominating spray on $Y$, Def.\ \ref{fiber-spray} below), 
was found by Gromov \cite{Gromov:OP}. A presumably even
weaker condition, {\em subellipticity} (Def.\ \ref{def:SS}),
was introduced in \cite{FF:subelliptic}.    

If $Y$ enjoys BOP or POP, then the corresponding Oka property also 
holds for sections of any holomorphic fiber bundle $Z\to X$ 
with fiber $Y$ over a Stein space $X$ \cite{FF:Kohn}.
See also Sect.\ \ref{subelliptic} below and 
the papers \cite{FF:surveyOka,Larusson1,Larusson2,Larusson3}. 

It is important to know which operations preserve Oka properties.
The following result was stated in \cite{FF:CAP} (remarks following Theorem 5.1), 
and more explicitly in \cite[Corollary 6.2]{FF:EOP}.
(See also \cite[Corollary 3.3.C']{Gromov:OP}.)

%
%
%
%
\begin{theorem}
\label{SES:ascend-descend}
Assume that $E$ and $B$ are complex manifolds. If $\pi \colon E\to B$ 
is a sub\-elliptic Serre fibration (Def.\ \ref{def:SSF} below), 
or a holomorphic fiber bundle with \POP\ fiber, then
the following hold:
\begin{itemize}
\item[(i)] If $B$ enjoys the parametric Oka property {\rm (\POP)}, then so does $E$.
\item[(ii)] If $E$ enjoys \POP\ for contractible parameter spaces $P$ 
(and arbitrary closed subspaces $P_0$ of $P$), then so does $B$.
\item[(iii)] If in addition $\pi \colon E\to B$ is a weak homotopy equivalence then 
\[
		{\rm POP\ of\ } E \Longrightarrow \rm{POP\ of\ } B.
\]
\end{itemize}
All stated implications  hold for a specific pair
$P_0\subset P$ of parameter spaces.
\end{theorem}

The proof Theorem \ref{SES:ascend-descend}, proposed in \cite{FF:EOP}, 
requires the parametric Oka property for sections of certain continuous families of 
subelliptic submersions. When Finnur L\'arusson asked for explanation 
and at the same time told me of his applications of this result 
\cite{Larusson4} (personal communication, December 2008),
I decided to write a more complete exposition. 
We prove Theorem \ref{SES:ascend-descend} in Sec.\ \ref{AD} 
as a consequence of Theorem \ref{SES:lifting-maps}.
This result should be compared with L\'arusson's \cite[Theorem 3]{Larusson4} 
where the map $\pi \colon E\to B$ is assumed to be an {\em intermediate fibration}
in the model category that he constructed.

\begin{corollary} 
\label{c1.6}
Let $Y=Y_m\to Y_{m-1}\to \cdots\to Y_0$, where each $Y_j$ is a complex manifold
and every map $\pi_j\colon Y_j\to Y_{j-1}$ $(j=1,2,\ldots,m)$ is a 
subelliptic Serre fibration, or a holomorphic fiber bundle with \POP\ fiber. 
Then the following hold:
\begin{itemize}
\item[(i)] 
If one of the manifolds $Y_j$ enjoys \BOP, or \POP\ with a contractible 
parameter space, then all of them do. 
\item[(ii)]
If in addition every $\pi_j$ is acyclic (a weak homotopy equivalence) and if 
$Y$ is a Stein manifold, then every manifold $Y_j$ in the tower satisfies
the implication \BOP\ $\Longrightarrow$ \POP.
\end{itemize}
\end{corollary}

\begin{proof} 
Part (i) is an immediate consequences of Theorem \ref{SES:ascend-descend}.
For (ii), observe that BOP of $Y_j$ implies BOP of $Y$ by
Theorem \ref{SES:ascend-descend} (i), applied with $P$ a singleton. 
Since $Y$ is Stein, BOP implies that $Y$ is elliptic (see Def.\ \ref{def:SS} below);
for the simple proof see \cite[Proposition 1.2]{FPrezelj:OP3} or \cite[3.2.A.]{Gromov:OP}.
By Theorem \ref{SES:OPS} below it follows that $Y$ also enjoys POP. 
By part (iii) of Theorem \ref{SES:ascend-descend}, POP descends from $Y=Y_m$
to every $Y_j$.
\end{proof}

\begin{remark}
\label{BiP}
A main remaining open problem is whether the implication 
\begin{equation}
\label{BOPimpliesPOP}
	\BOP \Longrightarrow \POP 
\end{equation}
holds for all complex manifolds.
By using results of this paper and of his earlier
works, F.\ L\'arusson proved this implication for a 
large class of manifolds, including all 
quasi-projective manifolds \cite[Theorem 4]{Larusson4}.
The main observation is that there exists an affine holomorphic fiber
bundle $\pi\colon E\to \P^n$ with fiber $\C^n$ whose total space $E$ is Stein;
since the map $\pi$ is acyclic and the fiber satisfies
POP, the implication (\ref{BOPimpliesPOP}) follows 
from Corollary \ref{c1.6} (ii) for any closed complex subvariety
$Y\subset \P^n$ (since the total space $E|_Y=\pi^{-1}(Y)$ is Stein). 
The same applies to complements of hypersurfaces in such $Y$;
the higher codimension case reduces to the hypersurface case by blowing up.
\qed \end{remark}

\section{Subelliptic submersions and Serre fibrations}
\label{subelliptic} 
A holomorphic map $h \colon Z\to X$ of complex spaces is a {\em holomorphic submersion} 
if for every point $z_0\in Z$ there exist an open neighborhood 
$V\subset Z$ of $z_0$, an open neighborhood $U\subset X$ of $x_0=h(z_0)$,
an open set  $W$ in a Euclidean space $\C^p$, and a biholomorphic map
$\phi\colon V\to U\times W$ such that $pr_1\circ \phi= h$,
where $pr_1 \colon U\times W\to U$ is the projection on the first factor.

\vskip -3mm
\[
	\xymatrix{ 
	\!\!\!\!\!\!\!\!\!\!\!\! Z\supset V \ar[d]^{h} \ar[r]^{\phi}  & U\times W \ar[d]^{pr_1} \\ 
	\!\!\!\!\!\!\!\!\!\!\!\!\! X\supset U  \ar[r]^{id}   & U}
\]
\smallskip

Each fiber $Z_x=h^{-1}(x)$ $(x\in X)$ of a holomorphic submersion 
is a closed complex submanifold of $Z$. 
A simple example is the restriction of a holomorphic fiber bundle projection
$E\to X$ to an open subset $Z$ of  $E$.

We recall from \cite{Gromov:OP,FF:subelliptic} 
the notion of a holomorphic spray and domination.

%
%
%
%
\begin{definition}
\label{fiber-spray}
Assume that $X$ and $Z$ are complex spaces and $h\colon Z\to X$
is a holomorphic submersion. For $x\in X$ let $Z_x=h^{-1}(x)$.
\begin{itemize}
\item[(i)] A {\em fiber-spray} on $Z$ is a triple $(E,\pi,s)$
consisting of a holomorphic vector bundle $\pi\colon E\to Z$ 
together with a holomorphic map $s\colon E\to Z$ such that
for each $z\in Z$ we have 
\[
    s(0_z)=z,\quad s(E_z) \subset Z_{h(z)}.
\]
\item[(ii)]  A spray $(E,\pi,s)$  is {\em dominating} at a point $z\in Z$
if its differential  
\[ (\d s)_{\, 0_z} \colon \T _{0_z} E \to \T _z Z\]
at the origin $0_z \in E_z=\pi^{-1}(z)$ maps the subspace $E_z$ of $\T _{0_z} E
$ surjectively onto the {\em vertical tangent space} $\VT_z Z=\ker \d h_z$.
The spray is {\em dominating} (on $Z$) 
if it is dominating at every point $z\in Z$.
\item[(iii)] 
A family of $h$-sprays $(E_j,\pi_j,s_j)$ $(j=1,\ldots,m)$ on $Z$
is dominating at the point $z\in Z$ if
\[
    (\d s_1)_{0_z}(E_{1,z}) + (\d s_2)_{0_z}(E_{2,z})\cdots
                     + (\d s_m)_{0_z}(E_{m,z})= \VT_z Z.
\]
If this holds for every $z\in Z$ then the family is {\em dominating} on $Z$.
\item[(iv)] A spray on a complex manifold $Y$ is a fiber-spray
associated to the constant map $Y\to point$.
\end{itemize}
\end{definition}

The simplest example of a spray on a complex manifold $Y$ is the flow
$Y\times \C\to Y$ of a $\C$-complete holomorphic vector field on $Y$.
A composition of finitely many such flows, with independent time variables, 
is a dominating spray at every point where the given collection of
vector fields span the tangent space of $Y$. Another example of a dominating
spray is furnished by the exponential map on a complex Lie group $G$,
translated over $G$ by the group multiplication.

The following notion of an {\em elliptic submersion} is due to 
Gromov \cite[Sect.\ 1.1.B]{Gromov:OP};
{\em subelliptic submersions} were introduced in \cite{FF:subelliptic}.
For examples see \cite{FF:subelliptic,FF:CAP,Gromov:OP}.

%
%
\begin{definition}                 
\label{def:SS}
A \holo\ submersion $h\colon Z\to X$ is said to be {\em elliptic} 
(resp.\ {\em subelliptic})
if each point $x_0\in X$ has an open \nbd\ $U\subset X$ such that
the restricted submersion $h\colon Z|_U \to U$ admits a dominating 
fiber-spray (resp.\ a finite dominating family of fiber-sprays).
A complex manifold $Y$ is elliptic (resp.\ subelliptic) if
the trivial submersion $Y\to point$ is such.
\end{definition}

The following notions appear in Theorem \ref{SES:ascend-descend}.

%
%
%
%
\begin{definition}
\label{def:SSF}
(a) A continuous map $\pi\colon  E\to B$ is
{\em Serre fibration} if it satisfies the homotopy lifting property 
for polyhedra (see {\rm \cite[p.\ 8]{Whitehead}}). 

(b) A holomorphic map $\pi\colon E\to B$ 
is an {\em elliptic Serre fibration} (resp.\ a {\em subelliptic Serre fibration}) 
if it is a surjective elliptic (resp.\ subelliptic) submersion and 
also a Serre fibration.
\end{definition}

The following result was proved in 
\cite{FF:Kohn} (see Theorems 1.4 and 8.3) by following the scheme 
proposed in \cite[Sect.\ 7]{FPrezelj:OP3}.
Earlier results include Gromov's Main Theorem 
\cite[Theorem 4.5]{Gromov:OP} (for elliptic submersions 
onto Stein manifolds, without interpolation),
\cite[Theorem 1.4]{FPrezelj:OP3} (for elliptic submersions
onto Stein manifolds), \cite[Theorem 1.1]{FF:subelliptic} 
(for subelliptic submersion), and \cite[Theorem 1.2]{FF:CAP} 
(for fiber bundles with POP fibers over Stein manifolds).

\begin{theorem}
\label{SES:OPS}
Let $h\colon Z\to X$ be a holomorphic submersion of a complex space $Z$ 
onto a Stein space $X$. Assume that $X$ is exhausted by Stein Runge domains 
$D_1\Subset D_2\Subset \cdots \subset X=\bigcup_{j=1}^\infty D_j$
such that every $D_j$ admits a stratification 
\begin{equation}
\label{strat}
	D_j = X_0\supset X_{1}\supset\cdots\supset X_{m_j} =\emptyset
\end{equation}
with smooth strata $S_k=X_k\bs X_{k+1}$ such that the restriction of 
$Z\to X$ to every connected component of each $S_k$ is a subelliptic submersion,
or a holomorphic fiber bundle with \POP\ fiber.
Then sections $X\to Z$ satisfy \POP. 
\end{theorem}

\begin{remark}
In previous papers 
\cite{FPrezelj:OP1,FPrezelj:OP2,FPrezelj:OP3,FF:subelliptic,FF:CAP,FF:EOP}
POP was only considered for pairs of parameter spaces  $P_0\subset P$ such that 

(*)  $P$ is a nonempty compact Hausdorff space, and $P_0$ 
is a closed subset of $P$ that is a strong 
deformation neighborhood retract (SDNR) in $P$. 

Here we dispense with the SDNR condition 
by using the Tietze extension theorem for maps into Hilbert spaces 
(see the proof of Proposition \ref{local-ext}).

Theorem \ref{SES:OPS} also hold when 
$P$ is a locally compact and countably compact Hausdorff space, 
and $P_0$ is a closed subspace of $P$. The proof requires
only a minor change of the induction scheme (applying the diagonal process).

On the other hand, all stated results remain valid if we restrict to
pairs $P_0\subset P$ consisting of finite polyhedra; this 
suffices for most applications.
\qed \end{remark}

Theorem \ref{SES:OPS} implies the following result concerning 
holomorphic liftings.

\begin{theorem}
\label{Lift}
Let $\pi \colon E\to B$ be a holomorphic submersion of a complex space $E$ 
onto a complex space $B$. Assume that $B$ admits a stratification 
$B=B_0\supset B_1\supset \cdots \supset B_m=\emptyset$ by 
closed complex subvarieties such that the 
restriction of $\pi$ to every connected component 
of each difference $B_j\bs B_{j+1}$ is a subelliptic submersion,
or a holomorphic fiber bundle with \POP\ fiber.

Given a Stein space $X$ and a holomorphic map $f\colon X\to B$,
every continuous lifting $F\colon X\to E$ of $f$ 
$(\pi\circ F=f)$ is homotopic through liftings of $f$ 
to a holomorphic lifting.
\end{theorem}

\vskip -4mm
\[
	\xymatrix{ & E \ar[d]^{\pi} \\ 
	           X \ar[r]^{f} \ar[ur]^{F} & B }
\]
\vskip 2mm

\begin{proof} 
Assume first that $X$ is finite dimensional. Then there is a stratification
$	X=X_0\supset X_{1}\supset\cdots\supset X_l=\emptyset$ by closed 
complex subvarieties, with smooth
differences $S_j=X_j\bs X_{j+1}$, such that each connected component $S$ 
of every $S_j$ is mapped by $f$ to a stratum $B_k\bs B_{k+1}$ for some $k=k(j)$.
The  pull-back submersion 
\[
		f^*E= \{(x,e)\in X\times E\colon f(x)=\pi(e)\} \longmapsto X
\]
then satisfies the
assumptions of Theorem \ref{SES:OPS} with respect to this
stratification of $X$. Note that liftings $X\to E$ of $f\colon X\to B$
correspond to sections $X\to f^*E$, and hence the result follows
from Theorem \ref{SES:OPS}.
The general case follows by induction over an exhaustion of
$X$ by an increasing sequence of relatively compact Stein
Runge domains in $X$.
\end{proof}

A fascinating application of Theorem \ref{Lift}
has recently been found by Ivarsson and Kutzschebauch 
\cite{Ivarsson-Kutzschebauch,Ivarsson-Kutzschebauch2}
who solved the following {\em Gromov's Vaserstein problem}:

\begin{theorem}
\label{Ivar-Kut} 
{\rm (Ivarsson and Kutzschebauch 
\cite{Ivarsson-Kutzschebauch,Ivarsson-Kutzschebauch2})}
Let $X$ be a finite dimensional reduced Stein space and let 
$f\colon X\to {\rm SL}_m(\mathbb{C})$ be a null-homotopic holomorphic mapping.
Then there exist a natural number $N$ and holomorphic mappings 
$G_1,\dots, G_{N}\colon X\to \mathbb{C}^{m(m-1)/2}$ 
(thought of as lower resp.\ upper triangular matrices) such that
\[
		f(x) = \left(\begin{array}{cc} 1  &  0 \cr G_1(x) & 1 \cr \end{array} \right)  
		\left(\begin{array}{cc} 1 & G_2(x) \cr 0 & 1 \cr \end{array} \right)  \ldots 
		\left(\begin{array}{cc} 1 & G_N(x)\cr 0 & 1 \cr \end{array} \right)
\]
is a product of upper and lower diagonal unipotent matrices.
(For odd $N$ the last matrix has $G_N(x)$ 
in the lower left corner.)
\end{theorem}

In this application one takes $B=SL_m(\C)$, $E=(\C^{m(m-1)/2})^N$, and 
$\pi\colon E\to B$ is the map
\[
		\pi(G_1,G_2,\ldots,G_N)= 
		\left(\begin{array}{cc} 1  &  0 \cr G_1 & 1 \cr \end{array} \right)  
		\left(\begin{array}{cc} 1 & G_2 \cr 0 & 1 \cr \end{array} \right)  \ldots 
		\left(\begin{array}{cc} 1 & G_N\cr 0 & 1 \cr \end{array} \right).
\]
Every null-homotopic holomorphic map $f\colon X\to B=SL_m(\C)$ 
admits a continuous lifting $F\colon X\to E$ for a suitably chosen $N\in \N$
(Vaserstein \cite{Vaser}), and the goal is to deform $F$ 
to a holomorphic lifting $G=(G_1,\ldots,G_N) \colon X\to E$.
This is done inductively by applying Theorem \ref{Lift}
to auxiliary submersions obtained by composing 
$\pi$ with certain row projections. Stratified elliptic submersions 
naturally appear in their proof.

\section{Convex Approximation Property}
\label{Sec:CAP}
In this section we recall from \cite{FF:CAP} 
a characterization of Oka properties in terms of 
an Oka-Weil approximation property for entire
maps $\C^n\to Y$. 

Let $z=(z_1,\ldots,z_n)$, $z_j=x_j+\I\, y_j$, 
be complex coordinates on $\C^n$. 
Given numbers $a_j,b_j >0$ $(j=1,\ldots,n)$ we set
\begin{equation}
\label{cube}
    Q =\{z\in \C^n\colon |x_j| \le a_j,\ |y_j|\le b_j,\ j=1,\ldots,n\}.
\end{equation}

\begin{definition}
\label{special}
A {\em special convex set} in $\C^n$
is a compact convex set of the form
\begin{equation}
\label{special-convex}
    K =\{z\in Q \colon y_n \le \phi(z_1,\ldots,z_{n-1},x_n)\},
\end{equation}
where $Q$ is a cube (\ref{cube}) and
$\phi$ is a continuous concave function with values
in $(-b_n,b_n)$. Such $(K,Q)$ is called a
{\em special convex pair} in $\C^n$.
\end{definition}

%
%
%
%
\begin{definition}
\label{def:CAP}
A complex manifold $Y$ enjoys the {\em Convex Approximation Property}
{\rm (CAP)} if every holomorphic map $f\colon K\to Y$
on a special convex set $K \subset Q \subset \C^n$ (\ref{special-convex})
can be approximated, uniformly on $K$, by holomorphic maps
$Q \to Y$. 

$Y$ enjoys the {\em Parametric Convex Approximation Property} {\rm (PCAP)}
if for every special convex pair $(K,Q)$ and for every pair of parameter 
spaces $P_0\subset P$ as in Def.\ \ref{POP}, a map $f\colon Q \times P\to Y$ 
such that $f_p=f(\cdotp,p)\colon Q\to Y$ is holomorphic for every $p\in P_0$,
and is holomorphic on $K$ for every $p\in P$, can be approximated 
uniformly on $K\times P$ by maps $\wt f\colon Q\times P\to Y$
such that $\wt f_p$ is holomorphic on $Q$ for all $p\in P$,
and $\wt f_p=f_p$ for all $p\in P_0$.
\end{definition}

The following characterization of the Oka property 
was found in \cite{FF:CAP,FF:EOP} (for Stein source
manifolds), thereby answering a question of 
Gromov \cite[p.\ 881, 3.4.(D)]{Gromov:OP}.
For the extension to Stein source spaces see \cite{FF:Kohn}.

\begin{theorem} 
\label{CAP}
For every complex manifold we have 
\[   
	\BOP\ \Longleftrightarrow \CAP, \qquad \POP \Longleftrightarrow \PCAP.
\]
\end{theorem}

\begin{remark}
The implication $\PCAP \Longrightarrow \POP$
also holds for a specific pair of (compact, Hausdorff) 
parameter spaces as is seen from the proof in \cite{FF:CAP}. 
More precisely, if a complex manifold $Y$
enjoys PCAP for a certain pair $P_0\subset P$, then
it also satifies POP for that same pair.
\qed\end{remark}

\section{A Parametric Oka Principle for liftings}
\label{SES;liftings}
In this section we prove the main result of this paper,
Theorem \ref{SES:lifting-maps}, which generalizes Theorem \ref{Lift}
to families of holomorphic maps.
We begin by recalling the relevant terminology from \cite{FPrezelj:OP3}.

%
%
\begin{definition}
\label{P-section}
Let $h\colon Z\to X$ be a holomorphic map of complex spaces,
and let $P_0 \subset P$ be topological spaces.
\begin{itemize}
\item[(a)] A {\em $P$-section} of $h \colon Z\to X$ is a
continuous map $f\colon X\times  P\to Z$ such that
$f_p=f(\cdotp,p)\colon X\to Z$ is a section of $h$ for each 
$p\in P$. Such $f$ is {\em holomorphic} if 
$f_p$ is holomorphic on $X$ for each fixed $p\in P$.
If $K$ is a compact set in $X$ and if $X'$ is a closed 
complex subvariety of $X$, then $f$ 
is {\em holomorphic on $K\cup X'$} if there is an open set 
$U\subset X$ containing $K$ such that the restrictions 
$f_p|_U$ and $f_p|_{X'}$ are holomorphic for every $p\in P$.
\item[(b)] A {\em homotopy of $P$-sections} is a continuous 
map $H\colon X\times P\times [0,1]\to Z$ such
that $H_t=H(\cdotp,\cdotp,t) \colon X\times P\to Z$ is a
$P$-section for each $t\in [0,1]$. 
\item[(c)] A {\em $(P,P_0)$-section} of $h$ is a $P$-section
$f\colon X\times  P\to Z$ such that $f_p=f(\cdotp,p)\colon X\to Z$
is holomorphic on $X$ for each $p\in P_0$.
A $(P,P_0)$-section is holomorphic on a subset $U\subset X$
if $f_p|_U$ is holomorphic for every $p\in P$.
\item[(d)] A $P$-map $X\to Y$ to a complex space $Y$ is a map
$X\times P\to Y$. Similarly one defines $(P,P_0)$-maps and
their homotopies.
\end{itemize}
\end{definition}

\begin{theorem}
\label{SES:lifting-maps}
Assume that $E$ and $B$ are complex spaces and $\pi \colon E\to B$ 
is a subelliptic submersion (Def.\ \ref{def:SSF}), or a holomorphic 
fiber bundle with {\rm POP} fiber (Def.\ \ref{POP}). 
Let $P$ be a compact Hausdorff space and
$P_0$ a closed subspace of $P$. 
Given a Stein space $X$, a compact $\cO(X)$-convex subset $K$ of $X$,
a closed complex subvariety $X'$ of $X$, a holomorphic 
$P$-map $f \colon X\times P \to B$, 
and a $(P,P_0)$-map $F\colon X\times P \to E$ that is 
a $\pi$-lifting of $f$ $(\pi\circ F=f)$ and is holomorphic on 
(a neighborhood of) $K$ and on the subvariety $X'$, 
there exists a homotopy of liftings 
$F^t\colon X\times P \to E$ of $f$ $(t\in [0,1])$ 
that is fixed on $(X\times P_0)\cup (X'\times P)$, that approximates 
$F=F^0$ uniformly on $K\times P$, and such that $F^1_p$ is holomorphic 
on $X$ for all $p\in P$.

If in addition $F$ is holomorphic in a neighborhood of $K\cup X'$
then the homotopy $F^t$ can be chosen such that it agrees with
$F^0$ to a given finite order along $X'$. 
\end{theorem}

\vskip -6mm
\[
	\xymatrix{ & E \ar[d]^{\pi} \\ 
	           X\times P \ar[r]_{f} \ar[ur]^{F_t} & B }
\]

\begin{definition}
\label{POPmaps}
A map $\pi \colon E\to B$ satisfying the conclusion of 
Theorem \ref{SES:lifting-maps} is said to enjoy the
parametric Oka property (c.f.\ L\'arusson \cite{Larusson2,Larusson3,Larusson4}).
\qed \end{definition}

\begin{proof}
The first step is a reduction to the graph case. 
Set $Z=X\times E$, $\wt Z=X\times B$, 
and let $\wt \pi \colon Z\to \wt  Z$ denote the map 
\[
		\wt \pi(x,e)=(x,\pi(e)),\quad x\in X,\ e\in E.
\]
Then $\wt \pi$ is a subelliptic submersion, resp.\ 
a holomorphic fiber bundle with POP fiber.
Let $\wt h\colon \wt Z= X\times B\to X$ 
denote the projection onto the first factor,
and let $h=\wt h\circ \wt \pi\colon Z\to X$.
To a $P$-map $f\colon X\times P \to B$ we associate
the $P$-section $\wt f(x,p)=(x,f(x,p))$ 
of $\wt h\colon \wt Z\to X$.
Further, to a lifting $F\colon X\times P\to E$ of $f$ 
we associate the $P$-section $\wt F(x,p)=(x,F(x,p))$ of $h\colon Z\to X$. 
Then $\wt \pi \circ \wt F = \wt f$.
This allows us to drop the tilde's on $\pi$, $f$ and $F$ 
and consider from now on the following situation:
\begin{itemize}
\item[(i)]    $Z$ and $\wt Z$ are complex spaces, 
\item[(ii)]   $\pi \colon Z \to \wt  Z$ is  a subelliptic submersion, or a holomorphic fiber 
bundle with \POP\ fiber,
\item[(iii)]  $\wt h \colon \wt Z\to X$ is a holomorphic map onto a Stein space $X$,
\item[(iv)]   $f\colon X\times P\to \wt Z$ is a holomorphic $P$-section of $\wt h$, 
\item[(v)]    $F\colon X\times P\to Z$ is a holomorphic $(P,P_0)$-section
of $h= \wt h\circ \pi \colon Z \to X$ such that $\pi\circ F=f$, and
$F$ is holomorphic on $K\cup X'$.
\end{itemize}
We need to find a homotopy $F^t\colon X\times P\to Z$ $(t\in[0,1])$
consisting of $(P,P_0)$-sections of $h \colon Z \to X$ such that $\pi \circ F^t=f$
for all $t\in[0,1]$, and 
\begin{itemize}
\item[($\alpha$)] $F^0=F$, 
\item[($\beta$)] $F^1$ is a holomorphic $P$-section, and
\item[($\gamma$)] for every $t\in [0,1]$, $F^t$ is holomorphic on $K$,
it is unifomly close to $F^0$ on $K\times P$, and
it agrees with $F^0$ on $(X\times P_0)\cup (X'\times P)$.
\end{itemize}
\vskip -3mm
\[
	\xymatrix{ & Z \ar[d]^{\pi} \\ 
	           X \times P \ar[r]_{f} \ar[ur]^{F^t} & \wt Z}
\]

Set $f_p=f(\cdotp,p)\colon X\to \wt Z$ for $p\in P$. The image $f_p(X)$ 
is a closed Stein subspace of $\wt Z$ that is biholomorphic to $X$
(since $\wt h\circ f$ is the identity on $X$). 

When $P=\{p\}$ is a singleton, there is only one section $f=f_p$,
and the desired conclusion follows by applying Theorem \ref{SES:OPS}
to the restricted submersion $\pi\colon Z|_{f(X)} \to f(X)$.

In general we consider the family of restricted 
submersions $Z|_{f_p(X)}\to f_p(X)$  $(p\in P)$.
The proof of the parametric Oka principle 
\cite[Theorem 1.4]{FPrezelj:OP2} requires certain modifications 
that we now explain. It suffices to obtain a homotopy 
$F^t$ of liftings of $f$ over a relatively compact subset 
$D$ of $X$ with $K\subset D$; the proof is then finished by 
induction over an exhaustion of $X$.
The initial step is provided by the following proposition.
(No special assumption is needed on the submersion 
$\pi\colon Z\to \wt Z$ for this result.)

\begin {proposition}
\label{local-ext}
{\rm (Assumptions as above)}
Let $D$ be an open relatively compact set in $X$ with $K\subset D$.
There exists a homotopy of liftings of $f$ over $D$ 
from $F=F^0|_{D\times P}$ to a lifting $F'$ such that 
properties ($\alpha$) and ($\gamma$) hold for $F'$, 
while $(\beta)$ is replaced by
\begin{itemize}
\item[($\beta$')] $F'_p$ is holomorphic on $D$ for all $p$ in a neighborhood 
$P'_0\subset P$ of $P_0$.
\end{itemize}
\end{proposition}

The existence of such local holomorphic extension $F'$ is used at several 
subsequent steps. We postpone the proof of the proposition 
to the end of this section and continue with the 
proof of Theorem \ref{SES:lifting-maps}.
Replacing $F$ by $F'$ and $X$ by $D$, we assume from now on that 
$F_p$ is holomorphic on $X$ for all $p\in P'_0$ (a neighborhood of $P_0$). 

Assume for the sake of discussion that $X$ is a Stein manifold,
that $X'=\emptyset$, and that $\pi \colon Z\to \wt Z$ is a 
subelliptic submersion. (The proof in the fiber bundle case is 
simpler and will be indicated along the way. The case when $X$
has singularities or $X'\ne \emptyset$ uses the 
induction scheme from \cite{FF:Kohn}, but the details
presented here remain unchanged.)
It suffices to explain the following:

\smallskip
{\bf Main Step:} Let $K\subset L$ be compact strongly pseudoconvex 
domains in $X$ that are $\cO(X)$-convex.
Assume that $F^0= \{F^0_p\}_{p\in P}$ is a $\pi$-lifting of $f= \{f_p\}_{p\in P}$
such that $F^0_p$ is holomorphic on $K$ for all $p\in P$, 
and $F^0_p$ is holomorphic on $X$ when $p\in P'_0$. 
Find a homotopy of liftings $F^t  = \{F^t_p\}_{p\in P}$ $(t\in[0,1])$ 
that are holomorphic on $K$, uniformly close to $F^0$ on $K\times P$, 
the homotopy is fixed for all $p$ in a neighborhood of $P_0$, and $F^1_p$ 
is holomorphic on $L$ for all $p\in P$.
\smallskip

Granted the Main Step, a solution satisfying the conclusion
of Theorem \ref{SES:lifting-maps} is then obtained by induction 
over a suitable exhaustion of $X$.

\smallskip
\noindent {\em Proof of Main Step.}
We cover the compact set 
$\bigcup_{p\in P} f_p(\overline{L\bs K})\subset \wt Z$
by open sets $U_1,\ldots, U_N\subset \wt Z$ 
such that every restricted submersion $\pi\colon Z|_{U_j}\to U_j$ 
admits a finite dominating family of $\pi$-sprays.
In the fiber bundle case we choose the sets $U_j$ such that $Z|_{U_j}$
is isomorphic to the trivial bundle $U_j\times Y \to U_j$
with POP fiber $Y$.

Choose a Cartan string $\cA=(A_0,A_1,\ldots,A_n)$ in $X$ 
\cite[Def.\ 4.2]{FPrezelj:OP2} such that $K=A_0$ and $L=\bigcup_{j=0}^n A_j$. 
The construction is explained in \cite[Corollary 4.5]{FPrezelj:OP2}: 
It suffices to choose each of the compact  sets $A_k$ 
to be a small strongly pseudoconvex domain such that 
$\left(\bigcup_{j=0}^{k-1} A_j,A_k\right)$ is a Cartan pair
for all $k=1,\ldots,n$. In addition, we choose the sets  
$A_1,\ldots, A_n$ small enough such that $f_p(A_j)$ is contained 
in one of the sets $U_l$ for every $p\in P$ and $j=1,\ldots, n$. 

We cover $P$ by compact subsets $P_1,\ldots,P_m$ 
such that for every $j=1,\ldots,m$ and $k=1,\ldots, n$, 
there is a neighborhood $P'_j\subset P$ of $P_j$ such that the set 
$
		\bigcup_{p\in P'_j} f_p(A_k) 
$
is contained in one of the sets $U_l$. 

As in \cite{FPrezelj:OP2} 
we denote by $\cK(\cA)$ the {\em nerve complex} of $\cA=(A_0,A_1,\ldots,A_n)$,
i.e., a combinatorial simplicial complex consisting of all
multiindices $J=(j_0,j_1,\ldots,j_k)$, with 
$0\le j_0<j_1<\cdots < j_k \le n$, such that
\[
    A_J=A_{j_0}\cap A_{j_1}\cap\cdots\cap A_{j_k} \ne \emptyset.
\]
Its {\em geometric realization}, $K(\cA)$, is a finite polyhedron 
in which every multiindex $J=(j_0,j_1,\ldots,j_k)\in \cK(\cA)$
of length $k+1$ determines a closed $k$-dimensional face
$|J|\subset K(\cA)$, homeomorphic to the standard $k$-simplex in $\R^k$, 
and every $k$-dimensional face of $K(\cA)$ is of this form. 
The face $|J|$ is called the {\em body} (or {\em carrier})
of $J$, and $J$ is the {\em vertex scheme} of $|J|$.
Given $I,J\in \cK(\cA)$ we have $|I|\cap |J|=|I\cup J|$. 
The vertices of $K(\cA)$ correspond to the individual sets $A_j$ in $\cA$,
i.e., to singletons $(j) \in \cK(\cA)$. 
(See \cite{Hurewicz-Wallman} or \cite{Spanier} for 
simplicial complexes and polyhedra.)

Given a compact set $A$ in $X$, we denote by $\Gamma_\cO(A,Z)$ 
the space of all sections of $h\colon Z\to X$ that are holomorphic over 
some unspecified open neighborhood $A$ in $Z$, in the sense of germs at $A$. 

Recall that a {\em holomorphic $\cK(\cA,Z)$-complex} 
\cite[Def.\ 3.2]{FPrezelj:OP2} is a continuous family of 
holomorphic sections
\[
  F_* = \{F_J \colon |J| \to \Gamma_\cO(A_J,Z),\ \ J \in \cK(\cA) \}
\]
satisfying the following compatibility conditions:
\[
    I,J\in \cK(\cA),\ I\subset J \Longrightarrow
      F_J(t) = F_I(t)|_{A_J} \ (\forall t\in |I|).
\]
Note that 
\begin{itemize}
\item $F_{(k)}$ a holomorphic section over (a neighborhood of) $A_k$, 
\item $F_{(k_0,k_1)}$ is a homotopy of holomorphic sections over 
$A_{k_0}\cap A_{k_1}$ connecting $F_{(k_0)}$ and $F_{(k_1)}$, 
\item $F_{(k_0,k_1,k_2)}$ is a triangle of homotopies with vertices 
$F_{(k_0)},F_{(k_1)},F_{(k_2)}$ and sides $F_{(k_0,k_1)},F_{(k_0,k_2)},F_{(k_1,k_2)}$,
etc. 
\end{itemize}
Similarly one defines a continuous $\cK(\cA,Z)$-complex.

A $\cK(\cA,Z;P)$-complex is defined in an obvious way by 
adding the parameter $p\in P$. It can be  viewed as a 
$\cK(\cA,Z)$-complex of $P$-sections of $Z\to X$, 
or as a family of $\cK(\cA,Z)$-complexes
depending continuously on the parameter $p\in P$.
Similarly, a $\cK(\cA,Z;P,P_0)$-complex is a $\cK(\cA,Z;P)$-complex 
consisting of holomorphic sections (over the set $L=\bigcup_{j=0}^n A_j$)
for the parameter values $p\in P_0$. The terminology of 
Def.\ \ref{P-section} naturally applies to complexes of sections.

By choosing the sets $A_1,\ldots, A_n$ sufficiently small
and by shrinking the neighborhood $P'_0$ (furnished by
Proposition \ref{local-ext}) around $P_0$ if necessary 
we can deform $F=F^0$ to a holomorphic $\cK(\cA,Z;P,P'_0)$-complex 
$F_{*,*} = \{F_{*,p}\}_{p\in P}$ such that 

\begin{itemize}
\item every section in $F_{*,p}$ projects by $\pi\colon Z\to \wt Z$ to the section $f_p$
(such $F_{*,*}$ is called a {\em lifting} of the holomorphic 
$P$-section $f=\{f_p\}_{p\in P}$),
\item  $F_{(0),p}$ is the restriction to $A_0=K$ 
of the initial section $F^0_p$, and 
\item for $p\in P'_0$, every section in $F_{*,p}$ is 
the restriction of $F^0_p$ to the appropriate subdomain
(i.e., the deformation from $F^0$ to $F_{*,*}$ is fixed over $P'_0$).
\end{itemize}

A completely elementary construction of such  
{\em initial holomorphic complex} $F_{*,*}$ can be 
found in \cite[Proposition 4.7]{FPrezelj:OP2}.

\begin{remark}
We observe that, although the map $h=\wt h\circ \pi \colon Z\to X$ 
is not necessarily a submersion 
(since the projection $\wt h\colon \wt Z\to X$ may have singular
fibers), the construction in \cite{FPrezelj:OP2} still
applies since we only work with the fiber component 
of $F_p$ (over $f_p$) with respect to the submersion $\pi\colon Z \to \wt Z$.
All lifting problems locally reduce to working with functions. 
\qed \end{remark}

The rest of the construction amounts to finitely many 
homotopic modifications of the complex $F_{*,*}$.
At every step we collapse one of the cells in the complex
and obtain a family (parametrized by $P$) 
of holomorphic sections
over the union of the sets that determine the cell. 
In finitely many steps we obtain a family of 
{\em constant complexes} $F^1=\{F^1_p\}_{p\in P}$,
that is, $F^1_p$ is a holomorphic section of $Z\to X$ over $L$. 
This procedure is explained in \cite[Sect.\ 5]{FPrezelj:OP2}
(see in particular Proposition 5.1.). The additional 
lifting condition is easily satisfied at every step of the construction. 
In the end, the homotopy of complexes from $F^0$ 
to $F^1$ is replaced by a homotopy of constant complexes,
i.e., a homotopy of liftings $F^t$ of $f$ that consist 
of sections over $L$ (see the conclusion
of proof of Theorem 1.5 in \cite[p.\ 657]{FPrezelj:OP2}).

Let us describe more carefully the main step --- collapsing a 
segment in a holomorphic complex. (All substeps in collapsing 
a cell reduce to collapsing a segment, each time 
with an additional parameter set.)

We have a special pair $(A,B)$ of compact sets contained in $L\subset X$, 
called a {\em Cartan pair} \cite[Def.\ 4]{FPrezelj:OP3}, 
with $B$ contained in one of the  
sets $A_1,\ldots,A_n$ in our Cartan string $\cA$.
(Indeed, $B$ is the intersection of some of these sets.)
Further, we have an additional compact parameter set $\wt P$
(which  appears in the proof) and 
families of holomorphic sections of $h\colon Z\to X$, 
$a_{(p,\tilde p)}$ over $A$ and 
$b_{(p,\tilde p)}$ over $B$, depending continuously
on $(p,\tilde p)\in P\times \wt P$ and projecting by 
$\pi\colon Z\to \wt Z$ to the section $f_p$.
For $p\in P'_0$ we have $a_{(p,\tilde p)}=b_{(p,\tilde p)}$ over $A\cap B$.
These two families are connected  over $A\cap B$ 
by a homotopy of holomorphic sections 
$b^t_{(p,\tilde p)}$ $(t\in[0,1])$ such that
\[
	b^0_{(p,\tilde p)}=a_{(p,\tilde p)}, 
	\quad 
	b^1_{(p,\tilde p)}=b_{(p,\tilde p)}, 
	\quad \pi\circ b^t_{(p,\tilde p)} =f_p
\]
hold for each $p\in P$ and $t\in [0,1]$, and the homotopy is fixed
for $p\in P'_0$. These two families are joined into a family 
of holomorphic sections $\wt a_{(p,\tilde p)}$ over $A\cup B$,
projecting by $\pi$ to $f_p$. The deformation 
consists of two substeps:
\begin{enumerate}
\item  by applying the Oka-Weil theorem \cite[Theorem 4.2]{FPrezelj:OP1}
over the pair $A\cap B\subset B$ we approximate the family 
$a_{(p,\tilde p)}$ sufficiently closely, uniformly on a neighborhood of $A\cap B$, 
by a family $\wt b_{(p,\tilde p)}$ of holomorphic sections over $B$;  
\smallskip
\item  assuming that the approximation in (1) 
is sufficiently close, we glue the families 
$a_{(p,\tilde p)}$ and $\wt b_{(p,\tilde p)}$ 
into a family of holomorphic sections $\wt a_{(p,\tilde p)}$ over $A\cup B$
such that $\pi\circ \wt a_{(p,\tilde p)}=f_p$.
\end{enumerate}

For Substep (2) we can use local holomorphic sprays 
as in \cite[Proposition 3.1]{FF:CAP}, or we apply
\cite[Theorem 5.5]{FPrezelj:OP1}.
The projection condition $\pi\circ \wt a_{(p,\tilde p)}=f_p$ 
is a trivial addition.

Substep (1) is somewhat more problematic as 
it requires a dominating family of $\pi$-sprays 
on $Z|_U$ over an open set $U\subset \wt Z$ to which the sections 
$b^t_{(p,\tilde p)}$ project. (In the fiber bundle case we need triviality
of the restricted bundle $Z|_U \to U$ and POP of the fiber.)
Recall that $B$ is contained in one of the sets $A_k$,
and therefore
\[
		\bigcup_{p\in P'_j} f_p(B) \subset \bigcup_{p\in P'_j} f_p(A_k) \subset U_{l(j,k)}.
\]
Since $\pi\circ b^t_{(p,\tilde p)} = f_p$ and 
$Z$ admits a dominating family of $\pi$-sprays over each 
set $U_l$, Substep (1) applies separately to each of the $m$ families 
\[
	\{b^t_{(p,\tilde p)} \colon p\in P'_j,\ \tilde p\in \wt P,\ t\in[0,1]\},
	\quad j=1,\ldots,m. 
\]

To conclude the proof of the Main Step we 
use the {\em stepwise extension method},
similar to the one in \cite[pp.\ 138-139]{FPrezelj:OP2}.
In each step we make the lifting holomorphic for the parameter values
in one of the sets $P_j$, keeping the homotopy fixed over the union 
of the previous sets.

We begin with $P_1$.
The above shows that the Main Step can be accomplished in finitely many
applications of Substeps (1) and (2), using the pair of 
parameter spaces $P_0 \cap P'_1 \subset P'_1$
(instead of $P_0\subset P$).
We obtain a homotopy of liftings $\{F^t_p\colon p\in P'_1, t\in [0,1]\}$ 
of $f_p$ such that $F^{1}_p$ is holomorphic on $L$ for all 
$p$ in a neighborhood of $P_1$, and $F^t_p=F^0_p$ for all $t\in[0,1]$ 
and all $p$ in a relative neighborhood of $P_0\cap P'_1$ in $P'_1$.
We extend this homotopy to all values $p\in P$ by replacing $F^t_p$ by 
$F^{t\chi(p)}_p$, where $\chi\colon P\to[0,1]$ is a continuous 
function that equals one near $P_1$ and has support contained in $P'_1$. 
Thus $F^1_p$ is holomorphic on $L$ for all $p$ in a neighborhood 
$V_1$ of $P_0\cup P_1$, and $F^1_p=F^0_p$ for all $p$
in a neighborhood of $P_0$. 

We now repeat the same procedure with $F^1$ as the `initial' lifting of $f$, 
using the pair of parameter spaces $(P_0\cup P_1) \cap P'_2 \subset P'_2$.
We obtain a homotopy of liftings $\{F^t_p\}_{t\in[1,2]}$ of $f_p$
for $p\in P'_2$ such that the homotopy is fixed 
for all $p$ in a neighborhood of $(P_0\cup P_1) \cap P'_2$
in $P'_2$, and $F^2_p$ is holomorphic on $L$ for all $p$ 
in a neighborhood of $P_0\cup P_1\cup P_2$ in $P$.

In $m$ steps of this kind we get a homotopy 
$\{F^t\}_{t\in [0,m]}$ of liftings of $f$ such that $F^m_p$ 
is holomorphic on $L$ for all $p\in P$, and the 
homotopy is fixed in a neighborhood of $P_0$ in $P$. 
It remains to rescale the parameter interval $[0,m]$ back to $[0,1]$.

This concludes the proof in the special case when $X$ is a Stein manifold
and $X'=\emptyset$. In the general case we follow the induction scheme
in the proof of the parametric Oka principle for stratified
fiber bundles with POP fibers in \cite{FF:Kohn}; 
Cartan strings are now used inside the smooth strata.

When $\pi \colon Z\to \wt Z$ is a fiber bundle,
we apply the one-step approximation and gluing procedure as in \cite{FF:CAP},
without having to deal with holomorphic complexes. 
The Oka-Weil approximation theorem in Substep (1) is replaced by 
POP of the fiber. 
\end{proof}

\begin{proof}[Proof of Proposition \ref{local-ext}]
We begin by considering the special case
when $\pi\colon Z=\wt Z\times \C \to\wt Z$ is a trivial line bundle.
We have $F_p=(f_p,g_p)$ where $g_p$ is a holomorphic
function on $X$ for $p\in P_0$, and is holomorphic on
$K\cup X'$ for all $p\in P$.
We replace $X$ by a relatively compact subset containing $\bar D$
and consider it as a closed complex subvariety of a Euclidean space $\C^N$. 
Choose bounded pseudoconvex domains $\Omega\Subset \Omega'$ 
in $\C^N$ such that $\bar D \subset \Omega\cap X$.

By \cite[Lemma 3.1]{FPrezelj:OP3} there exist bounded
linear extension operators
\begin{eqnarray*}
		 S  \colon H^\infty(X\cap \,\Omega') 
		 &\longrightarrow& H^2(\Omega) = L^2(\Omega)\cap \cO(\Omega), \\
		 S' \colon H^\infty(X'\cap \,\Omega') &\longrightarrow&  H^2(\Omega),
\end{eqnarray*}
such that $S(g)|_{X\cap\, \Omega}=g|_{X\cap\, \Omega}$, and likewise for $S'$.
(In \cite{FPrezelj:OP3} we obtained an extension operator into 
$H^\infty(\Omega)$, but the Bergman space appeared as an intermediate
step. Unlike the Ohsawa-Takegoshi extension theorem
\cite{OhsawaT}, this is a soft result depending on the Cartan extension 
theorem and some functional analysis; the price is shrinking
of the domain.) Set
\[
		h_p = S(g_p|_{X\cap \,\Omega'}) - S'(g_p|_{X'\cap \,\Omega'}) \in H^2(\Omega), 
		\qquad p\in P_0.
\]
Then $h_p$ vanishes on $X'$, and hence it belongs to the closed
subspace $H^2_{X'}(\Omega)$   consisting of all
functions in $H^2(\Omega)$ that vanish on $X'\cap \Omega$.
Since these are Hilbert spaces, the generalized Tietze extension
theorem (a special case of Michael's convex selection theorem;
see \cite[Part C, Theorem 1.2, p.\ 232]{Repovs} or \cite{Dowker,Lee})
furnishes a continuous extension of the map $P_0 \to  H^2_{X'}(\Omega)$,
$p \to h_p$, to a map $P\ni p \to \wt  h_p\in H^2_{X'}(\Omega)$. Set 
\[
		G_p = \wt h_p + S'(g_p|_{X'\cap \,\Omega'}) \in H^2(\Omega),\qquad p\in P.
\]
Then 
\[ 
	G_p|_{X'\cap\, \Omega} = g_p|_{X'\cap\, \Omega} \ (\forall p\in P), \quad 
 G_p|_{X\cap\, \Omega} = g_p|_{X\cap\, \Omega}    \ (\forall p\in P_0).
\]
This solves the problem, except that $G_p$ should
approximate $g_p$ uniformly on $K$. Choose 
holomorphic functions $\phi_1,\ldots,\phi_m$ on $\C^N$ that generate the 
ideal sheaf of the subvariety $X'$ at every point in $\Omega'$. 
A standard application of Cartan's Theorem B shows that
in a neighborhood of $K$ we have
\[
		g_p = G_p + \sum_{j=1}^m \phi_j \, \xi_{j,p} 
\]
for some holomorphic functions $\xi_{j,p}$ in a neighborhood of $K$,
depending continuously on $p\in P$ and vanishing identically on $X$ 
for $p\in P_0$. (See e.g.\ \cite[Lemma 8.1]{FPrezelj:OP2}.)

Since the set $K$ is $\cO(X)$-convex, and hence polynomially convex in $\C^N$,
an extension of the Oka-Weil approximation theorem  (using a bounded 
linear solution operator for the $\dibar$-equation, given for instance  
by H\"ormander's $L^2$-methods \cite{Hor} or by integral kernels)	
furnishes functions $\wt \xi_{j,p} \in\cO(\Omega)$,
depending continuously on $p\in P$, such that $\wt \xi_{j,p}$ approximates $\xi_{j,p}$
as close as desired uniformly on $K$, and it vanishes 
on $X\cap \Omega$ when $p\in P_0$. Setting 
\[
			\wt  g_p = G_p + \sum_{j=1}^m  \phi_j \, \wt \xi_{j,p},\quad p\in P
\]
gives the solution. 
This proof also applies to vector-valued maps by applying it componentwise.

The general case reduces to the special case by using that 
for every $p_0 \in P_0$, the Stein subspace $F_{p_0}(X)$ 
(resp.\ $f_{p_0}(X)$) admits an open Stein neighborhood in $Z$
(resp.\ in $\wt Z$) according to a theorem of Siu \cite{Dem,Siu}.
Embedding these neighborhoods in Euclidean spaces and using holomorphic
retractions onto fibers of $\pi$ (see \cite[Proposition 3.2]{FF:Kohn}),
the special case furnishes neighborhoods $U_{p_0}\subset U'_{p_0}$
of $p_0$ in $P$ and a $P$-section $F'\colon \bar D\times P\to Z$,
homotopic to $F$ through liftings of $f$, such that
\begin{itemize}
\item[(i)]   $\pi\circ  F'_p=f_p$ for all $p\in P$,  
\item[(ii)]  $F'_p$ is holomorphic on $\bar D$ when $p\in U_{p_0}$,
\item[(iii)]  $F'_p=F_p$ for $p\in P_0 \cup (P\bs U'_{p_0})$,  
\item[(iv)] $F'_p|_{X'\cap D}=F_p|_{X'\cap D}$ for all $p\in P$, and 
\item[(v)]  $F'$ approximates $F$ on $K\times P$.
\end{itemize}
(The special case is first used for  parameter values 
$p$ in a neighborhood $U'_{p_0}$ of $p_0$;
the resulting family of holomorphic maps $\bar D\times U'_{p_0} \to \C^N$ 
is then patched with $F$ by using 
a cut-off function $\chi(p)$ with support in $U'_{p_0}$
that equals one on a neighborhood $U_{p_0}$ of $p_0$, 
and applying holomorphic retractions onto the fibers of $\pi$.) 
In finitely many steps of this kind we complete the proof.
\end{proof}

\begin{remark}
One might wish to extend Theorem \ref{SES:lifting-maps} 
to the case when $\pi\colon E\to B$ is a {\em stratified} 
subelliptic submersion, or a stratified fiber bundle with POP fibers.
The problem is that the induced stratifications 
on the pull-back submersions $f_p^* E\to X$ may change 
discontinuously with respect to the parameter $p$. 
Perhaps one could get a positive result by assuming
that the stratification of $E\to B$ is suitably compatible
with the variable map $f_p \colon X \to B$. 
\qed \end{remark}

Recall (Def.\ \ref{POPmaps}) that a holomorphic map 
$\pi\colon E\to B$ satisfies POP if the conclusion of 
Theorem \ref{SES:lifting-maps} holds. We show that 
this is a local property.

\begin{theorem}
\label{localization}
{\bf (Localization principle for POP)}
A holomorphic submersion $\pi\colon E\to B$ of a complex space 
$E$ onto a complex space $B$ satisfies \POP\ if and only if
every point $x\in B$ admits an open neighborhood $U_x\subset B$
such that the restricted submersion $\pi\colon E|_{U_x}\to U_x$
satisfies \POP.
\end{theorem}

\begin{proof}
If $\pi\colon E\to B$ satisfies POP then clearly so does 
its restriction to any open subset $U$ of $B$.

Conversely, assume that $B$ admits an open covering
$\cU=\{U_\alpha\}$ by open sets such that every 
restriction $E|_{U_\alpha}\to U_\alpha$ enjoys POP.
When proving POP for $\pi\colon E\to B$, a typical step
amounts to choosing small compact sets $A_1,\ldots,A_n$ in 
the source (Stein) space $X$ such that, for a given 
compact set $A_0\subset X$, $\cA=(A_0,A_1,\ldots,A_n)$
is a Cartan string. We can choose the sets $A_1,\ldots,A_n$
sufficiently small such that each map 
$f_p\colon X\to B$ in the given family sends each $A_j$
into one of the sets $U_\alpha\in\cU$.

To the string $\cA$ we associate a $\cK(\cA,Z;P,P_0)$-complex
$F_{*,*}$ which is then inductively deformed into a holomorphic
$P$-map $\wt F \colon \bigcup_{j=0}^n A_j \times P \to E$ 
such that $\pi\circ \wt F=f$. The main step in the inductive procedure 
amounts to patching a pair of liftings over a Cartan pair
$(A',B')$ in $X$, where the set $B'$ is contained in one of the sets 
$A_1,\ldots, A_n$ in the Cartan string $\cA$.
This is subdivided into substeps (1) and (2) 
(see the proof of Theorem \ref{SES:lifting-maps}). 
Only the first of these substeps, which requires a
Runge-type approximation property, is a nontrivial condition 
on the submersion $E\to B$. It is immediate from the definitions 
that this approximation property holds 
if there is an open set $U\subset B$ containing the image $f_p(B')$ 
(for a certain set of parameter values $p\in P$)
such that the restricted submersion $E|_U \to U$ satisfies 
POP.  In our case this is so since we have insured that 
$f_p(B')\subset f_p(A_j)\subset U_\alpha$ for some
$j\in\{1,\ldots,n\}$ and $U_\alpha\in\cU$.
\end{proof}

\section{Ascent and descent of the parametric Oka property}
\label{AD}
In this section we prove Theorem \ref{SES:ascend-descend} stated in Sec.\ \ref{Sec:Oka}.

\smallskip
\noindent {\em Proof of (i):} 
Assume that $B$ enjoys POP (which is equivalent to PCAP).
Let $(K,Q)$ be a special convex pair in $\C^n$ (Def.\ \ref{special}),
and let $F \colon Q \times P\to E$ be a $(P,P_0)$-map that is
holomorphic on $K$ (Def.\ \ref{P-section}). 

Then $f=\pi\circ F\colon Q\times P\to B$ is a $(P,P_0)$-map that is holomorphic on $K$. 
Since $B$ enjoys POP, there is a holomorphic $P$-map $g\colon Q\times P\to B$
that agrees with $f$ on $Q\times P_0$ and is uniformly close to 
$f$ on a neighborhood of $K\times P$ in $\C^n\times P$. 

If the latter approximation is close enough, there exists a 
holomorphic $P$-map $G\colon K\times P\to E$ such that 
$\pi\circ G=g$, $G$ approximates $F$ on $K\times P$, 
and $G=F$ on $K\times P_0$. To find such lifting of $g$,
we consider graphs of these maps (as in the proof of 
Theorem \ref{SES:lifting-maps}) and apply a holomorphic retraction 
onto the fibers of $\pi$ \cite[Proposition 3.2]{FF:Kohn}.

Since $G=F$ on $K\times P_0$, we can extend $G$ to 
$(K\times P) \cup (Q\times P_0)$ by setting $G=F$ 
on $Q\times P_0$. 

Since $\pi\colon E\to B$ is a Serre fibration and $K$ is a strong deformation
retract of $Q$ (these sets are convex),  $G$ extends to a continuous 
$(P,P_0)$-map $G\colon Q\times P\to E$ such that $\pi\circ G=g$.
The extended map remains holomorphic on $K$.

By Theorem \ref{SES:lifting-maps} there is a homotopy of liftings
$G^t\colon Q\times P\to E$ of $g$ $(t\in[0,1])$ which is fixed on
$Q \times P_0$ and is holomorphic and uniformly close to 
$G^0=G$ on $K\times P$. The holomorphic $P$-map $G^1\colon Q\times P\to E$
then satisfies the condition in Def.\ \ref{def:CAP}
relative to $F$. This proves that $E$ enjoys PCAP and hence POP.

\smallskip
\noindent {\em Proof of (ii):} 
Assume that $E$ enjoys POP. Let $(K,Q)$ be a special convex pair, and let 
$f \colon Q \times P\to B$ be a $(P,P_0)$-map that is holomorphic on $K$. 
Assuming that $P$ is contractible, the Serre fibration property 
of $\pi\colon E\to B$ insures the existence of a continuous $P$-map 
$F\colon Q\times P\to E$ such that $\pi\circ F=f$.
(The subset $P_0$ of $P$ does not play any role here.)
Theorem \ref{SES:lifting-maps} furnishes a homotopy 
$F^t\colon Q\times P\to E$ $(t\in [0,1])$ such that 
\begin{itemize}
\item[(a)]  $F^0=F$,
\item[(b)]  $\pi\circ F^t=f$ for each $t\in[0,1]$, and 
\item[(c)]  $F^1$ is a $(P,P_0)$-map that is holomorphic on $K$. 
\end{itemize}
This is accomplished in two steps: 
We initially apply Theorem \ref{SES:lifting-maps}  
with $Q\times P_0$ to obtain a homotopy $F^t\colon Q\times P_0\to E$ 
$(t\in [0,\frac{1}{2}])$, satisfying properties (a) and (b) above, 
such that $F^{1/2}_p$ is holomorphic on $Q$ for all $p\in P_0$. 
For trivial reasons this homotopy extends continuously to all values $p\in P$.
In the second step we apply Theorem \ref{SES:lifting-maps} over $K\times P$,
with $F^{1/2}$ as the initial lifting of $f$ and keeping the homotopy fixed 
for $p\in P_0$ (where it is already holomorphic), 
to get a homotopy $F^t$ $(t\in[\frac{1}{2},1])$ such that $\pi\circ F^t=f$ 
and $F^1_p$ is holomorphic on $K$ for all $p\in P$. 

Since $E$ enjoys POP, $F^1$ can be approximated uniformly
on $K\times P$ by holomorphic $P$-maps $\wt F \colon Q\times P\to E$
such that $\wt F=F^1$ on $Q\times P_0$. 
Then 
\[
		\wt  f=\pi\circ \wt F\colon Q\times P\to B
\]
is a holomorphic $P$-map that agrees with $f$ on $Q\times P_0$
and is close to $f$ on $K\times P$. 

This shows that $B$ enjoys PCAP for any contractible 
(compact, Hausdorff) parameter space $P$ and for any closed 
subspace $P_0$ of $P$. Since the implication PCAP$\Longrightarrow$POP
in Theorem \ref{CAP} holds for each specific pair $(P_0,P)$
of parameter spaces, we infer that $B$ also enjoys POP for such
parameter pairs. This completes the proof of (ii).

\smallskip
\noindent {\em Proof of (iii):} 
Contractibility of $P$ was used in the proof of (ii) to lift 
the map $f \colon Q \times P\to B$ to a map $F\colon Q \times P\to E$.
Such a lift exists for every topological space if $\pi\colon E\to B$ is 
a weak homotopy equivalence. This is because a Serre fibration 
between smooth manifolds is also a Hurewicz fibration 
(by Cauty \cite{Cauty}), and a weak homotopy equivalence 
between them is a homotopy  equivalence by the Whitehead Lemma.
\qed

\medskip
{\em Acknowledgement.}
I express my sincere thanks to Finnur L\'arusson for his
questions which led to this paper, and for very helpful 
discussions and remarks. 
I also thank Petar Pave\v si\'c and Du\v san Repov\v s
for advice on Tietze extension theorem used in the 
proof of Proposition \ref{local-ext}.

\bibliographystyle{amsplain}

\end{document}